\documentclass[11pt,oneside,reqno]{amsart}
\usepackage[latin1]{inputenc}
\usepackage[english]{babel}
\usepackage{mathrsfs}
\usepackage{indentfirst}
\usepackage{paralist}
\usepackage{amssymb}
\usepackage{amsthm}
\usepackage{amsmath}
\usepackage{amscd}
\usepackage{graphicx}
\usepackage[colorlinks=true]{hyperref}
\usepackage[margin=1in]{geometry}
\usepackage{cite}
\usepackage{lineno} 
\usepackage{enumitem} 
\usepackage{multicol}
\usepackage{xcolor}
\definecolor{ForestGreen}{rgb}{0.13, 0.55, 0.13}

\newtheorem{theorem}{Theorem}[section] 
\newtheorem*{theorem*}{Main Result}
\newtheorem{thm*}{Known result}
\newtheorem{corollary}[theorem]{Corollary} 

\newtheorem{proposition}[theorem]{Proposition}
\newtheorem{definition}[theorem]{Definition}

\theoremstyle{definition}

\theoremstyle{remark}
\newtheorem{remark}[theorem]{Remark}

\numberwithin{equation}{section} 
\newcommand{\eqlab}[1]{\begin{equation}  \begin{aligned}#1 \end{aligned}\end{equation}} 
\newcommand{\bgs}[1]{\begin{equation*} \begin{aligned}#1\end{aligned}\end{equation*}} 
\newcommand{\sys}[2][]{\begin{equation*}#1  \left\{\begin{aligned}#2\end{aligned}\right.\end{equation*}}

\DeclareMathOperator{\tail}{Tail}
\newcommand{\R}{\ensuremath{\mathbb{R}}}
\newcommand{\Rn}{\ensuremath{\mathbb{R}^n}}
\newcommand{\N}{\ensuremath{\mathbb{N}}}

\DeclareMathOperator*{\esssup}{ess\,sup}
\DeclareMathOperator*{\essinf}{ess\,inf}
\DeclareMathOperator*{\esslimsup}{ess\,limsup}
\DeclareMathOperator*{\essliminf}{ess\,liminf}
\newcommand{\Ll}{\mathcal L}
\newcommand{\Co}{\mathcal C}
\newcommand{\E}{\mathcal E}

\newcommand{\W}{\mathcal W}

\DeclareMathOperator{\Per}{Per}

\DeclareMathOperator{\diam}{diam}

\newcommand{\eps}{\varepsilon}
\renewcommand{\le}{\leqslant}
\renewcommand{\leq}{\leqslant}
\renewcommand{\ge}{\geqslant}
\renewcommand{\geq}{\geqslant}
\renewcommand{\emptyset}{\varnothing}

\DeclareMathOperator{\PV}{\mbox{\normalfont P.V.}}

\title[Continuity of $s$-minimal functions]{Continuity of $s$-minimal functions}
\begin{document}
\author[C. Bucur]{Claudia Bucur}
\address[C.\ Bucur]{\newline\indent Dipartimento di Matematica Federigo Enriques
	\newline\indent
	Universit\`a degli Studi di Milano
	\newline\indent
	Via Saldini 50, 20133 Milano, Italia}
\email{\href{mailto:claudia.bucur@unimi.it}{claudia.bucur@unimi.it}} 

\author[S. Dipierro]{Serena Dipierro}
\address[S.\ Dipierro]{\newline\indent Department of Mathematics
and Statistics
\newline\indent
University of Western Australia
\newline\indent
35 Stirling Hwy, Crawley WA 6009, Australia}
\email{\href{mailto:serena.dipierro@uwa.edu.au}{serena.dipierro@uwa.edu.au}}

\author[L. Lombardini]{Luca Lombardini}
\address[L.\ Lombardini]{\newline\indent
Institut f\"ur Analysis und Scientific Computing 
\newline\indent
Technische Universit\"at Wien
\newline\indent
Wiedner Hauptstra{\ss}e 8--10, 1040 Vienna, Austria}
\email{\href{mailto:luca.lombardini@tuwien.ac.at}{luca.lombardini@tuwien.ac.at}}
\author[E. Valdinoci]{Enrico Valdinoci}
\address[E.\ Valdinoci]{\newline\indent Department of Mathematics
and Statistics
\newline\indent
University of Western Australia
\newline\indent
35 Stirling Hwy, Crawley WA 6009, Australia}
\email{\href{mailto:enrico.valdinoci@uwa.edu.au}{enrico.valdinoci@uwa.edu.au}}

\thanks{CB has been supported by INdAM --  GNAMPA, project E53C23001670001. LL acknowledges the support of the Austrian Science Fund (FWF) through grants 10.55776/F65 and 10.55776/Y1292}

\begin{abstract}
We consider the minimization property of a Gagliardo-Slobodeckij seminorm which can be seen as the fractional counterpart of the classical problem of functions of least gradient and which is related to the minimization of the nonlocal perimeter functional.

We discuss continuity properties for this kind of problem. In particular, we show that, under natural structural assumptions, the minimizers are bounded and continuous in the interior of the ambient domain (and, in fact, also continuous up to the boundary under some mild additional hypothesis).

We show that these results are also essentially optimal, since in general the minimizer is not necessarily continuous across the boundary.
\end{abstract}
\maketitle

\section{Introduction}

The most intensively studied topic in the calculus of variations is probably the minimization of a functional in a suitable~$L^p$-type space. In this line of research, the case~$p=1$ is customarily ``exceptional'', due to the peculiar functional properties of~$L^1$-type spaces.

A framework of special interest is that provided by the ``functions of least gradient'', i.e. by the functions~$u$ minimizing the total variation of the vector-valued measure~$\nabla u$. This setting is important for at least two reasons: on the one hand, it provides the appropriate $L^1$-type functional framework for Dirichlet energies by formulating the question in the bounded variation sense; on the other hand, it relates the problem to another classical one, namely the theory of parametric minimal surfaces (the connection arising from to the coarea formula). See in particular~\cite{MR1246349, SWZ, J, mazcl} for classical results on functions of least gradient and their connections with minimal surfaces.

The recent literature has also considered the fractional counterpart of this type of problems, in the setting of a suitable minimization of a Gagliardo-Slobodeckij seminorm,
see in particular~\cite{bdlv, bdlvm}. These problems are relevant also in view of their connections with the nonlocal minimal surfaces introduced by Caffarelli, Roquejoffre, and Savin in~\cite{nms}: indeed, as established in~\cite{bdlv}, a function is a minimizer for the fractional Gagliardo-Slobodeckij seminorm of~$W^{s,1}$-type (with~$s\in(0,1)$) if and only if its level sets are minimizers for the corresponding fractional perimeter.
\medskip

The goal of this article is to continue the investigation of the minimizers of the $W^{s,1}$-type seminorm and prove (or disprove) their continuity and check if they satisfy some comparison principle.

In the classical case, these results were obtained in~\cite{MR1246349, SWZ}. The fractional framework is however structurally different, since calibration methods are at the moment not available, and the comparison principle for fractional minimal surfaces is a delicate issue which has been established only very recently (see~\cite{seo23}). Besides, fractional minimal surfaces are highly sensitive to their data ``at infinity'' in terms of their ``boundary stickiness'' features (see~\cite{boundary, graph, sticki, seograph, stickscrelle}) and this special characteristic also influences the type of assumptions required at the level of minimization of the Gagliardo-Slobodeckij seminorm. \medskip

The mathematical setting in which we work is the following.
Given~$s\in (0,1)$, let~$\Omega\subset \Rn$ be a bounded open set with Lipschitz boundary and define
\[
	 \W^{s,1}(\Omega):= \big\{ u\colon \Rn \to \R \; {\mbox{ s.t. }} \; u|_\Omega \in W^{s,1}(\Omega) \big\} \]
and, for~$\varphi \colon \Co \Omega \to \R$, 
  \[  \W_\varphi^{s,1}(\Omega):= \big\{ u\colon \Rn \to \R \; {\mbox{ s.t. }} \; u \in \W^{s,1}(\Omega), \; u=\varphi \mbox{ in } \Co \Omega \big\}.
 \]
Here above and in what follows, we use the notation~$\Co \Omega$
to denote the complement of~$\Omega$, namely~$\Co \Omega:=\R^n\setminus\Omega$.
Also, for every~$r>0$, we set
$$
\Omega_r:=\{x\in\Rn\; {\mbox{ s.t. }}\;\textrm{dist}(x,\Omega)<r\}.
$$
We also denote by~$Q(\Omega):=\R^{2n}\setminus(\Co \Omega)^2$.

Our main object of interest here is the set of minimizers in~$\W^{s,1}(\Omega)$, according to the following terminology:

\begin{definition}\label{defn11}
	We say that~$u\in \W^{s,1}(\Omega)$ is an $s$-minimal function in~$\Omega$ if
	\begin{equation}\label{GH2}
	\iint_{Q(\Omega)} \big(|u(x) -u(y)|-|v(x)  - v(y)|\big) 
	\frac{dx\, dy}{|x-y|^{n+s}} \leq 0,
	\end{equation}
	for any competitor~$v\in \W^{s,1}(\Omega)$ such that~$u=v$ almost everywhere in~$\Co \Omega$.
	\end{definition}
	
	In this framework, our first result states that $s$-minimal functions are bounded and continuous:
	
\begin{theorem}\label{exlev}
There exists~$\Theta = \Theta(n,s)>1$ such that the following statement holds true.

If~$\Omega\subset \Rn$ is a bounded open set with Lipschitz boundary such that
\eqlab{ \label{connect} \mbox{the set } \; 
 \overline{\Omega_D}\setminus \Omega  \mbox{ is connected} ,}
 for some~$D>\Theta \diam(\Omega)$, and
if $\varphi \colon \Co \Omega \to \R$ is such that 
	\begin{equation}\label{PHIA}
		\varphi \in  
		C(\overline{\Omega_{D} }\setminus \Omega),
	 \end{equation}
	 then any $s$-minimal function~$u\in \W^{s,1}_\varphi(\Omega)$ belongs to~$ C(\Omega)\cap L^\infty(\Omega)$. 
	 \end{theorem}
	 
	 The continuity up to the boundary of the domain is also established separately, according to the following result:
	 
	 \begin{theorem}\label{upb}
	Let~$\Omega \subset\R^n$ be a bounded open set with boundary of class~$C^2$ and~$u$ be an $s$-minimal function in~$\Omega$. 
	
If~$u|_{\Omega}\in C(\Omega)\cap L^\infty(\Omega)$
and~$u = \varphi$ almost everywhere in~$\Co \Omega$, with~$\varphi:  \Co \Omega \to \R$ such that~$\varphi\in C(\Omega_\delta\setminus\Omega)$ for some~$\delta>0$, then~$u|_{\Omega}$ can be extended to a function~$\bar{u} \in C(\overline{\Omega})$.
\end{theorem}

Gathering together Theorems~\ref{exlev} and~\ref{upb}, one obtains:

\begin{corollary}\label{contbdr}
There exists~$\Theta = \Theta(n,s)>1$ such that the following statement holds true.

If~$\Omega\subset \Rn$ is a bounded open set with~$C^2$ boundary such that~\eqref{connect} holds true, and
if~$\varphi \colon \Co \Omega \to \R$ is such that~\eqref{PHIA} holds true, 
then any $s$-minimal function~$u\in \W^{s,1}_\varphi(\Omega)$ belongs to~$C(\overline{\Omega})$. More precisely, $u|_\Omega\in C(\Omega)$ can be extended to a function~$\bar{u} \in C(\overline{\Omega})$.
\end{corollary}

We also stress that the connectedness assumption
in~\eqref{connect} is not merely technical and cannot be removed from Theorem~\ref{exlev},
as the next result points out:

\begin{theorem}\label{pro42}
There exist a bounded domain~$\Omega$ with Lipschitz
boundary and~$s_1\in(0,1)$ such that, for all~$s\in[s_1,1)$,
there exists an~$s$-minimal function
which is not continuous 
in~$\Omega$.
\end{theorem}

A more delicate issue, also in view of stickiness results which appear to be typical for nonlocal objects, is the continuity of $s$-minimal functions across the boundary. 
We exhibit some positive results in
Propositions~\ref{EXA-01} and~\ref{cont}, 
however we stress that discontinuities across the boundary  may arise. In particular, we have:

\begin{theorem}\label{LACK}
Let 
$$\Omega:=\big\{ (x,y) \in \R^2 \; {\mbox{ s.t. }} \; (x+1)^2+y^2<1\big\}$$ and let~$\varphi \colon \Co \Omega \to \R$ be such that~$\varphi \in C_c(\Co \Omega)$, $\mbox{supp } \varphi \subset B_1\setminus \Omega$ and~$\varphi(0,0)=1$.

Then, there exists some~$\widetilde s\in (0,1)$ such that, for all~$s\in(0,\widetilde s)$, any $s$-minimal function~$u_s\in \W_\varphi^{s,1}(\Omega)$ is not continuous across the boundary. 
\end{theorem}

We also present a comparison principle between maximum/minimum solutions:

\begin{theorem}\label{cp1} There exists~$\Theta = \Theta(n,s)>1$ such that the following statement holds true.

If~$\Omega\subset \Rn$ is a bounded open set with Lipschitz boundary such that \eqref{connect} holds true and if $\varphi_1, \varphi_2 \colon \Co \Omega \to \R$ satisfy \eqref{PHIA} and are such that \[\varphi_1 \geq \varphi_2,\]
		then
		\[ \overline{ u}_1\geq \overline{ u}_2\qquad{\mbox{and}}\qquad \underline{u}_1 \geq \underline{ u}_2,\]
		where~$\overline{u}_i$ and~$\underline{u}_i$, with~$i\in \{1,2\}$,
	are the maximum, respectively the minimum, $s$-minimal functions with respect to the exterior data~$\varphi_i$.
\end{theorem}

The rest of this paper is organized as follows.
In Section~\ref{sec:2prelim} we collect some notation, basic definitions
and preliminary results that will be used throughout the paper.

Section~\ref{sec:proey2hfqf} contains the proofs of the
continuity statements
in Theorems~\ref{exlev} and~\ref{upb}, while Section~\ref{sec4duweitiuw}
is devoted to examples of continuity and
lack of continuity across the boundary
and to the proof of Theorem~\ref{LACK}.

The necessity of the connectedness assumption~\eqref{connect}
in Theorem~\ref{exlev} is discussed in Section~\ref{sef7671fjhfjhq},
together with the example constructed in
Theorem~\ref{pro42}.

Section~\ref{fgiw4t734sec777}
is devoted to 
the comparison statement of Theorem~\ref{cp1}.

The paper concludes with Appendices~\ref{appe:po8372} and \ref{apboh}, in which we discuss  technical details pertaining to Theorem \ref{exlev} in the first appendix and Theorem~\ref{LACK} in the latter.

\section{Definitions and preliminaries} \label{sec:2prelim}

We use, along this paper, the following notations for measure theoretic interior, exterior and boundary. Given a measurable set~$E\subset\Rn$, we define its \emph{measure theoretic boundary} as
$$
\partial^- E:=\{x\in\Rn\;{\mbox{ s.t. }}\;0<|E\cap B_r(x)|< |B_r(x)|\mbox{ for every }r>0\}.
$$
This is the topological boundary of the set~$E^{(1)}$ of points of density~$1$ of~$E$, i.e.
\[ E^{(1)}:=\left\{ x\in \Rn \;{\mbox{ s.t. }} \; \lim_{r\to 0} \frac{|E\cap B_r(x)|}{|B_r(x)|} =1\right\}.\]  We recall that~$|E\Delta E^{(1)}|=0$ by the Lebesgue Differentiation Theorem.

The \emph{measure theoretic interior} and \emph{exterior} of~$E$
are defined respectively as
$$
E_{int}:=\{x\in\Rn\; {\mbox{ s.t. }}\;|E\cap B_r(x)|= |B_r(x)|\mbox{ for some }r>0\}
$$
and
$$
E_{ext}:=\{x\in\Rn\; {\mbox{ s.t. }}\;|E\cap B_r(x)|= 0\mbox{ for some }r>0\}.
$$
These are open sets and~$E_{int}$ is the topological interior of~$E^{(1)}$. Denote also
\[ \mbox{cl}(E) =E_{int} \cup \partial^- E\]
as the \emph{measure theoretic closure} of the set~$E$.
	\medskip
	
Regarding the Definition \ref{defn11}, we point out that an $s$-minimal function $u\in \W^{s,1}(\Omega)$ is well defined without a priori conditions on the exterior data (and this is due to a fractional Hardy inequality, see \cite{bdlv}). In~\cite[Theorem~1.5]{bdlv}, we proved the existence of an $s$-minimal function whenever the ``nonlocal tail'' of the exterior data in a large enough neighborhood of $\Omega$ is summable in~$\Omega$. To be more precise, there exists~$\Theta >1$ such that, whenever 
\begin{equation} \label{Thbe2w4} \| \tail_s(\varphi, \Omega_{\Theta \diam(\Omega)} \setminus \Omega; \cdot)\|_{L^1(\Omega)} :=\int_\Omega\bigg[
\int_{\Omega_{\Theta \diam(\Omega) } \setminus \Omega} \frac{ |\varphi(y)|}{|x-y|^{n+s}} \, dy
		 \bigg]\,dx <+\infty,
	 \end{equation}
then there exists an $s$-minimal function~$u \in \W_\varphi^{s,1}(\Omega)$. In particular, if~$\varphi \in L^\infty(\Omega_{\Theta d}\setminus \Omega)$, then~\eqref{Thbe2w4} holds true.

If~\eqref{Thbe2w4} stands and we consider the energy functional
 \eqlab{ \label{energy} 
 \E(u) :=\frac{1}{2}\iint_{Q(\Omega)} \frac{|u(x)-u(y)|}{|x-y|^{n+s}} \, dx\, dy,}
we see that~$\E(u)$ is finite for all~$u\in \W^{s,1}_\varphi(\Omega)$. 
 According to~\cite[Lemma~2.1]{bdlv}, we have that~$u$ is an $s$-minimal function if and only if~$u$
 is a minimizer of~$\E$ in~$\Omega$, i.e.
 \[ \E(u) \leq \E(v), \qquad \mbox{ for all } v \in \W^{s,1}_\varphi(\Omega) .\]
\smallskip

Furthermore, in~\cite{bdlv}, we discussed the connection between $s$-minimal functions and nonlocal minimal surfaces. 
 Nonlocal minimal surfaces were introduced in~\cite{nms}
as objects mimimizing a nonlocal perimeter. 
More precisely, given a fractional parameter~$s\in(0,1)$, the $s$-fractional perimeter of a measurable set~$E\subset\R^n$ in an open set~$\Omega\subset\R^n$ is defined as
\begin{equation}\label{defnotper}\begin{split}
\Per_s(E,\Omega)&:=\frac{1}{2}\iint_{Q(\Omega)}\frac{|\chi_E(x)-\chi_E(y)|}{|x-y|^{n+s}}dx\,dy\\
&
=\Ll_s(E\cap\Omega,\Omega\setminus E)+\Ll_s(E\cap\Omega,\Co E\setminus\Omega)+\Ll_s(\Omega\setminus E,E\setminus\Omega),
\end{split}\end{equation}
where the interaction~$\Ll_s$ of two measurable sets~$A$ and~$B\subset\R^n$ such that~$|A\Delta B|=0$ is given by
$$
\Ll_s(A,B):=\int_A\int_B\frac{dx\,dy}{|x-y|^{n+s}}.
$$
When~$\Omega=\R^n$, we simply have that~$\Per_s(E,\R^n)=\frac{1}{2}[\chi_E]_{W^{s,1}(\R^n)}$.

We say that~$E$ is an $s$-minimal set for the fractional perimeter in~$\Omega$
(and~$\partial E$ is a nonlocal minimal surface) if~$\Per_s(E,\Omega)<+\infty$
and
$$ \Per_s(E,\Omega)\le \Per_s(F,\Omega) $$
for all~$F\subset\R^n$ such that~$F\setminus \Omega=E\setminus\Omega$.
In this context, \cite[Theorem 1.3]{bdlv} gives that a function is $s$-minimal if and only if all of its level sets are $s$-minimal sets. 

We also recall that, as pointed out in~\cite{nms}, nonlocal minimal surfaces satisfy a fractional mean curvature equation, that is, if~$E$ is $s$-minimal for the
fractional perimeter in~$\Omega$, then,
for every~$x\in(\partial E)\cap\Omega$,
$$ H_s[E](x)
=0,$$ 
in the sense given by
\begin{equation}\label{mc} H_s[E](x) =\lim_{\rho \to 0} H_s^\rho[E](x), \qquad
 H_s^\rho[E](x)  = \int_{\Rn\setminus B_\rho(x)}\frac{\chi_{\R^n\setminus E}(y)-\chi_{E}(y)}{|x-y|^{n+s}}\,dx.
 \end{equation}

\section{Continuity up to the boundary
and proofs of Theorems~\ref{exlev}
and~\ref{upb}}\label{sec:proey2hfqf} 

In this section, we establish the continuity results stated in Theorems~\ref{exlev}
and~\ref{upb}.

\begin{proof}[Proof of Theorem~\ref{exlev}]
By~\cite[Theorem~4.8]{bdlv},  there exists an $s$-minimal function~$u \in \W_\varphi^{s,1}(\Omega)$. Also, from~\cite[formula~(4.33) and Theorem~4.4]{bdlv} we have that~$ u\in L^\infty(\Omega)$ and 
  \begin{equation}\label{INund:sd9i331} \esssup_{\Omega} u \leq \sup_{\Omega_{\Theta d}\setminus \Omega} \varphi\qquad{\mbox{and}}\qquad \essinf_{\Omega} u \geq \inf_{\Omega_{\Theta d}\setminus \Omega} \varphi.\end{equation}
As customary, by~$u$ being continuous
we mean that there exists a function~$\widetilde u \in C(\Omega)$ such that~$u = \widetilde u$ almost everywhere in~$\Omega$. This is equivalent to having
  \[\ell^-(x_0)= \ell^+(x_0) \quad \mbox{ for every } x_0\in \Omega,   \]
  		where
  		\begin{eqnarray*}
  		 \ell^-(x_0) &:= &\essliminf_{ x\to x_0} u(x) := \lim_{r \searrow 0} \essinf_{B_r(x_0)} u\\
  		 {\mbox{and }}\qquad \ell^+(x_0)&:=&
  		  \esslimsup_{ x\to x_0} u(x):=
  		  \lim_{r \searrow 0} \esssup_{B_r(x_0)} u
  		 .\end{eqnarray*}
  		
  		  	By~\eqref{INund:sd9i331} we have that 
  	\begin{equation}\label{31BIS}
  	 \inf_{\overline{\Omega_D}\setminus \Omega} \varphi \leq \ell^- \leq \ell^+\leq \sup_{\overline{\Omega_D}\setminus \Omega} \varphi.\end{equation}
  	 
We point out that the claim of Theorem~\ref{exlev} is proved
if we check that~$\ell^- (x)= \ell^+(x)$ for all~$x\in\Omega$.
To establish this, we suppose by contradiction that there exists some~$x_0\in \Omega$ such that \[  \ell^-(x_0)< \ell^+(x_0).\]
In light of this, we take~$t\in ( \ell^-(x_0), \ell^+(x_0))$ and
notice that this implies that there exists some~$\varepsilon>0$ such that 
  \begin{equation}
  \label{fg}
  \essinf_{B_\varepsilon(x_0)} u < t <\esssup_{B_\varepsilon(x_0)} u .
  \end{equation}   
  Moreover, from~\cite[Theorem~1.3]{bdlv}
  we have that~$E_t:=\{ u\geq t\}$ is $s$-minimal for the fractional perimeter
  in~$\Omega$. 
  
  We claim that 
  \begin{equation}\label{quo3ry8325ygekuwgk}\begin{split}
& {\mbox{there exists some~$r_0>0$ such that~$B_{r_0}(x_0)\subset \Omega$ and, for all~$r\in (0,r_0)$,}}\\
 & 0<|E_t \cap B_r(x_0)| <|B_r(x_0)|.\end{split}\end{equation}
  Indeed, the lack of these strict inequalities would imply that either~$u \geq t$ or~$u<t$ almost everywhere in~$B_r(x_0)$, for all~$r\in (0,r_0)$, which
  would contradict~\eqref{fg}. This establishes~\eqref{quo3ry8325ygekuwgk}.
  
We point out that the claim in~\eqref{quo3ry8325ygekuwgk}  
ensures that~$x_0\in \partial^- E_t$.

In light of these considerations, 
we have that, taking~$t_1<t_2 \in (\ell^-(x_0),\ell^+(x_0))$, then
  \[ x_0 \in\partial^- E_{t_1}\cap \partial^- E_{t_2}\]
  and~$E_{t_1}$ and~$E_{t_2}$ are $s$-minimal in~$B_{r_0}(x_0)$. But~\cite[Theorem~1.1]{seo23} implies that~$E_{t_1}=E_{t_2}$, and in particular 
  \[ \{ \varphi\geq t_1\} =\{\varphi \geq t_2\} \] (up to sets of measure zero), 
  which is in contradiction with
  the continuity of~$\varphi$ on a connected set and the fact that~$t_1$, $t_2 \in (\inf_{\overline{\Omega_D}\setminus \Omega} \varphi, \sup_{\overline{\Omega_D}\setminus \Omega} \varphi)$,
  thanks to~\eqref{31BIS}
  (see Remark~\ref{kk2} for full details about this technical point). 
 \end{proof}
 
We now establish boundary continuity of $s$-minimal functions:

\begin{proof}[Proof of Theorem~\ref{upb}]
	The argument is an adaptation of the proof of~\cite[Proposition~8.2]{teolu} and relies on the regularity theory for the obstacle problem for the~$s$-perimeter developed by Caffarelli, De Silva and Savin~\cite{obss}.
	
	To show that~$u$ is continuous up to the boundary of~$\Omega$, it is enough to prove that
	\begin{equation} \label{bdary_reg_claim_proof}
		\mbox{for every } x \in \partial \Omega, \mbox{ the limit } 
		\ell(x) := \lim_{\Omega \ni y\to x} u(y) \mbox{ exists and is finite}.
	\end{equation}
	Indeed, if this is the case, then it is easy to see that~$\ell \in C(\partial \Omega)$ and thus that the extension of~$u|_\Omega$ by~$\ell$ defines a continuous function in the whole~$\overline{\Omega}$.
	
	To prove~\eqref{bdary_reg_claim_proof}, we define
	\bgs{
		\ell^-(x):=\liminf_{\Omega \ni y \to x} u(y)\in\R
		\qquad\mbox{and}\qquad
		\ell^+(x):=\limsup_{\Omega \ni y \to x} u(y)\in\R,
	}
	for every~$x\in\partial\Omega$. These are indeed well-defined, since~$u|_{\Omega}\in C(\Omega)\cap L^\infty(\Omega)$. Claim~\eqref{bdary_reg_claim_proof} is then equivalent to showing that~$\ell^-(x)=\ell^+(x)$ for all~$x\in\partial\Omega$.
	
	We argue by contradiction and suppose that~$\ell^-(x_0) < \ell^+(x_0)$ at some~$x_0\in\partial\Omega$. Then, at least one between~$\ell^-(x_0)$ and~$\ell^+(x_0)$
	is different from~$\varphi(x_0)$. We assume that~$\ell^-(x_0)<\varphi(x_0)$, the other cases being proved in a similar way.
	
	Consider the level sets~$E_t:=\{u\geq t\}\subset\Rn$, which are $s$-minimal for the fractional perimeter
	in~$\Omega$ for every~$t\in\R$, by~\cite[Theorem~1.3]{bdlv}. We begin by observing that, for every~$t\in(\ell^-(x_0),\ell^+(x_0))$,
	\eqlab{\label{meas_bdary}
	0<\big|E_t\cap (B_r(x_0)\cap\Omega)\big|<|B_r(x_0)\cap\Omega|\qquad\mbox{for every }r>0.
}
	Indeed, if there exists~$r>0$ such that~$\big|E_t\cap (B_r(x_0)\cap\Omega)\big|=0$, then~$u< t$ in~$B_r(x_0)\cap\Omega$. 
	This implies that~$\ell^+(x_0)\leq t$, thus giving a contradiction. Similarly, $\big|E_t\cap (B_r(x_0)\cap\Omega)\big|=|B_r(x_0)\cap\Omega|$ for some~$r>0$
	would yield that~$\ell^-(x_0)\geq t$.
	This completes the proof of~\eqref{meas_bdary}, and gives additionally that~$x_0\in\partial^-E_t$. 
	
	Now we claim that \eqlab{\label{top_bdary}
		x_0\in\partial^-E_t\quad\mbox{and there exist }x_k^t\in\partial^-E_t\cap\Omega\, \mbox{ such that }\, x_k^t\xrightarrow{k\to+\infty}x_0.
	}
	To prove this, 
	we first observe that, by the regularity of~$\Omega$, there exists~$r_0>0$ such that~$\Omega\cap B_r(x_0)$ is a connected open set for every~$r\in(0,r_0)$. Then, notice that we can write~$\Omega\cap B_r(x_0)$ as the disjoint union
	$$
	\Omega\cap B_r(x_0)=\Big(\Omega\cap B_r(x_0)\cap (E_t)_{int}\Big)\cup\Big(\Omega\cap B_r(x_0)\cap (E_t)_{ext}\Big)\cup\Big(\Omega\cap B_r(x_0)\cap \partial^-E_t\Big).
	$$
	Suppose now that~$\Omega\cap B_r(x_0)\cap \partial^-E_t=\emptyset$. Then, the connectedness of~$\Omega\cap B_r(x_0)$ implies that either
	$$
	\Omega\cap B_r(x_0)\cap (E_t)_{int}=\emptyset,
	$$
	or
	$$
	\Omega\cap B_r(x_0)\cap (E_t)_{ext}=\emptyset.
	$$
	However, both eventualities are in contradiction with~\eqref{meas_bdary}.
	Therefore, $\Omega\cap B_r(x_0)\cap \partial^-E_t\not=\emptyset$ for every~$r\in(0,r_0)$, which completes the proof of~\eqref{top_bdary}.
	
	Consider now~$t_1$, $t_2\in(\ell^-(x_0),\varphi(x_0))$, with~$t_1<t_2$. By the definition of~$E_t$ and the continuity of~$u$ in~$\Omega$, we have that
	\eqlab{\label{inclusion}
	E_{t_2}\subset E_{t_1}\qquad\mbox{and}\qquad\big|E_{t_1}\setminus E_{t_2}\big|\geq \big|(E_{t_1}\setminus E_{t_2})\cap\Omega\big|>0.
	}
	Moreover, since~$x_0\in\partial^-E_{t_2}$,
	by the continuity of~$\varphi$ in~$\Omega_\delta\setminus\Omega$, we can find~$\varrho\in(0,\delta)$ small enough such that
	$$
	B_\varrho(x_0)\setminus\Omega\subset E_{t_2}\subset E_{t_1}.
	$$
	We can thus apply~\cite[Theorem~5.1]{graph} to conclude that the boundaries~$\partial E_{t_2}$ and~$\partial E_{t_1}$ are of class~$C^{1,\frac{1+s}{2}}$ in~$B_\varrho(x_0)$, up to considering a smaller~$\varrho$.
	Actually, the interior regularity of $s$-minimal sets ensures that these boundaries are of class~$C^\infty$ in~$\Omega\cap B_\varrho(x_0)$
	(see~\cite{bootstrap}).
	
	Therefore, the Euler-Lagrange equation of $s$-minimal sets implies that~$H_s[E_{t_i}](x^{t_i}_k)=0$ for every~$k\in\N$ and~$i=1,2$, with~$x^{t_i}_k$ as in~\eqref{top_bdary}. Since the regularity of class~$C^{1,\frac{1+s}{2}}$ of the boundaries is enough to guarantee the continuity of the fractional mean curvature (see e.g.~\cite{mattheorem}), we obtain that
	$$
	H_s[E_{t_i}](x_0)=\lim_{k\to+\infty}H_s[E_{t_i}](x^{t_i}_k)=0,\qquad\mbox{for }i=1,2.
	$$
	
	Then, we can conclude the proof by making use of the strong comparison principle. Indeed, we see that
	\bgs{
	0 &= H_s[E_{t_2}](x_0)-H_s[E_{t_1}](x_0)\\
	&
	=\PV\int_{\Rn}\frac{\chi_{\R^n\setminus E_{t_2}}(x)
	-\chi_{E_{t_2}}(x)-\chi_{\R^n\setminus E_{t_1}}(x)+\chi_{E_{t_1}}(x)}{|x-x_0|^{n+s}}\,dx\\
	&
	=2\PV\int_{\Rn}\frac{\chi_{E_{t_1}\setminus E_{t_2}}(x)}{|x-x_0|^{n+s}}\,.
	}
Now we remark that
the integrand~$\chi_{E_{t_1}\setminus E_{t_2}}\ge0$, thanks to
the first claim in~\eqref{inclusion}. Therefore, this computation
shows that~$\chi_{E_{t_1}\setminus E_{t_2}}=0$, and thus~$E_1=E_2$,
in contradiction
	with the second statement in~\eqref{inclusion}.
	
	Hence, we have proved that~$\ell^-(x)=\ell^+(x)$ for every~$x\in\partial\Omega$, which entails~\eqref{bdary_reg_claim_proof}, thus concluding the argument.
\end{proof}

\section{Lack of continuity across the boundary
and proof of Theorem~\ref{LACK}}\label{sec4duweitiuw}

We build in this section examples of both continuity and lack of it across the boundary.

The first example,  quite elementary, is one of continuity across the boundary and uniqueness. 

\begin{proposition} \label{EXA-01}
	Let~$\Omega$ be a bounded open set and let~$\varphi \colon \Co \Omega \to \R$ be a constant function, say~$\varphi\equiv c\in \R$.
	
	Then, there exists a unique $s$-minimal function~$u\in \W_\varphi^{s,1}(\Omega)$, which is the constant function~$u\equiv c $.
\end{proposition} 

\begin{proof}
The fact that there exists an $s$-minimal function is a 
consequence of~\cite[Theorem~1.5]{bdlv}.
Also, \cite[Theorem~1.3]{bdlv} gives that the set~$\{ u \geq t\} $ is $s$-minimal for the fractional perimeter
in~$\Omega$ with respect to the exterior datum~$\{\varphi \geq t\}$.

However, thanks to~\cite[Theorem~1.7]{bdlv}, for all~$t\leq c$, we have that~$\{ \varphi \geq t\} =\Co \Omega$, and thus~$\{ u \geq t\} \cap \Omega= \Omega$, and, for all~$t>c$, we have that~$\{ \varphi \geq t\} =\emptyset$, and hence~$\{ u \geq t\}\cap \Omega =\emptyset$. It follows that~$u\equiv c$, concluding the proof. 
\end{proof}

Notice the similarity of the given example with non-local minimal sets. Indeed, if the exterior datum is the half-plane, the unique $s$-minimal set is the half-plane itself. 

What is more, for non-local minimal surfaces, a very interesting feature of stickiness arises (see~\cite{boundary, seograph}). 
In particular, in~\cite[Theorem~1.4]{boundary}
one constructs an exterior datum in~$\R^2$ looking at two compactly supported bumps, with support away from~$\Omega$, that
push the $s$-minimal surface to stick to the boundary in~$\Omega$
(see~\cite[Figure~4]{boundary}). 
Our second example, showcasing again continuity across the boundary, is in contrast with this construction for nonlocal minimal surfaces.

\begin{proposition}\label{cont}
Let $\Omega \subset \Rn$ be a bounded and connected open set with $C^2$ boundary, and let~$\varphi \in C_c(\Omega)$ be such that~$\varphi \geq 0$ and~$\mbox{supp } \varphi \subset \Co \Omega_\delta$, for some~$\delta>0$. 

Then, there exists~$\widetilde s\in (0,1)$ such that, for all~$s\in(0,\widetilde s)$, if~$u_s \in \W^{s,1}_\varphi(\Omega)$ is an $s$-minimal function, then~$u_s\equiv0$ in~$\Omega$. 
\end{proposition}

\begin{proof}
Denote by~$K:= \mbox{supp } \varphi $ and~$B:=\max_{K} \varphi$. By~\cite[Theorem~4.4]{bdlv}, we have 
that~$0\leq u_s(x) \leq B$ for all~$x\in \Omega$.
 
The conclusion follows from Proposition \ref{propb1}, recalling the notation \eqref{alphaf}. Since for all~$t\in (0,B]$ it holds that ~$\mathcal E_t:= \{ \varphi \geq t \} \subset K$, while $\overline \alpha(K)=0$  
and~$K$ does not completely surround~$\Omega$, 
there is some $s_0:=s_0(n,\Omega, K)\in (0,1/2)$ such that for all~$s<s_0$,
\[ E_s^t\cap \Omega:= \{ u_s\geq t\} \cap \Omega =\emptyset.\]
It follows that~$u _s\equiv 0$ in~$\Omega$, concluding the proof.
\end{proof}

We now address an example of lack of continuity across the boundary: 

\begin{proof}[Proof of Theorem~\ref{LACK}] The proof is basically the same as that of Proposition~\ref{cont}. 
	Denote by~$B:= \max_{\overline{B_1}}\varphi$
	and notice that~$0\leq u_s\leq B$ in~$\Omega$. For all~$t\in (0,B]$, as in the proof of Proposition~\ref{cont}, 
	there exists some~$\widetilde s$ independent of $t$ such that, for all~$s\in(0,\widetilde s)$, 
\[ \{ u\geq t\} \cap \Omega =\emptyset.\]
This gives that~$u\equiv0$ in~$\Omega$, and therefore~$u(0,0)\neq \varphi(0,0)$. 
\end{proof}

An example of a function~$\varphi$ satisfying the assumptions
of Theorem~\ref{LACK}
is~$\varphi(x,y) :=\big(1-\sqrt{x^2+y^2}\big)_+$. Notice that a finite number of compactly supported bumps, with support away from~$\Omega$, can be added without changing the conclusion. Also, more general examples can be constructed with the same procedure.

\section{Necessity of the connectedness assumption in~\eqref{connect}
and proof of Theorem~\ref{pro42}}\label{sef7671fjhfjhq}

We recall that the classical perimeter of~$E$ in~$\Omega$ is
$$
\Per(E,\Omega)=[\chi_E]_{BV(\Omega)}=\mathcal H^{n-1}(\partial^*E\cap\Omega),
$$
where~$\partial^*E$ denotes the reduced boundary of~$E$. If~$E\subset\R^n$ is a bounded set with finite classical perimeter, then
\eqlab{\label{eq:asympt_global}
\lim_{s\to1}(1-s)\Per_s(E,\R^n)=c_n\Per(E,\R^n),	
}	
for some dimensional constant~$c_n>0$, see~\cite{davila} and also~\cite{BBM}.

We consider now the closed set~$K\subset\R^2$ with Lipschitz boundary defined as
$$
K:=\left\{(x_1,x_2)\in\R^2\;{\mbox{ s.t. }}\;x^2_1+x_2^2\leq 1\mbox{ and }x_2
\leq 5|x_1|\right\},
$$
and the open set~$\Omega:=B_2\setminus K$, see
Figure~\ref{figagg}. Notice that~$\Omega$ is bounded and connected, with Lipschitz (disconnected) boundary.
In particular, the assumption in~\eqref{connect} is violated.

\begin{figure}[h!]
	\includegraphics[scale=0.6]{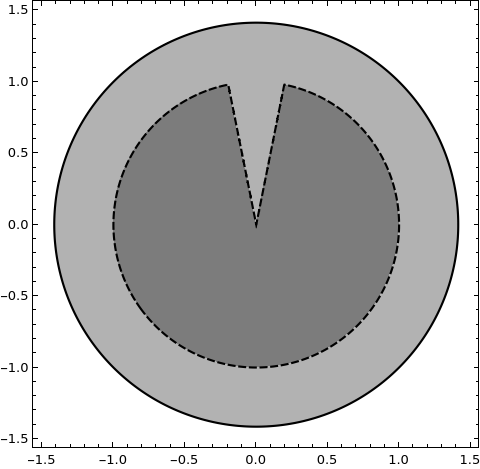}
	\caption{The sets~$K$ and~$\Omega$.}
	\label{figagg}
\end{figure}

We point out that
\bgs{
	\Per(K,\R^2)>\Per(B_1,\R^2).
}
Thus, as a consequence of~\eqref{eq:asympt_global}, we know that there exists~$s_1=s_1(K)\in(0,1)$ such that
\eqlab{\label{eq:strict_esti_sets_per}
	\Per_s(K,\R^2)>\Per_s(B_1,\R^2)
}
for every~$s\in[s_1,1)$.

We then obtain the following result:

\begin{proposition}\label{pro41}
	Let~$s\in[s_1,1)$ and let~$E\subset\R^2$ be any set minimizing~$\Per_s(\,\cdot\,,\Omega)$ among all measurable sets~$F\subset\R^2$ such that~$F\setminus\Omega=K$.
	
	Then,
\begin{equation}\label{des9}
	0<|E\cap\Omega|<|\Omega|.
\end{equation}
\end{proposition}

\begin{proof}
	We observe that
	$$
	|E\cap\Omega|=0 \quad{\mbox{ if and only if }}\quad|E\Delta K|=0
	$$
	and
	$$
	|E\cap\Omega|=|\Omega|\quad{\mbox{ if and only if }}\quad|E\Delta B_2|=0.
	$$
	Thus, to establish the desired result in~\eqref{des9}, it is enough to prove that
	\eqlab{\label{eq:strict_comp}
	\Per_s(B_1,\Omega)<\min\big\{\Per_s(K,\Omega),\Per_s(B_2,\Omega)\big\}.
	}
	For this, 
	recall the notation in~\eqref{defnotper} and notice that
	\bgs{
	\Per_s(B_1,\Omega)&=\Ll_s(B_1\setminus K,B_2\setminus B_1)+\Ll_s(B_1\setminus K,\Co B_2)+\Ll_s(B_2\setminus B_1,K)\\
	&
	=\Ll_s(B_1\setminus K,\Co B_1)+\Ll_s(K,\Co B_1)-\Ll_s(K,\Co B_2)\\
	&
	=\Ll_s(B_1,\Co B_1)-\Ll_s(K,\Co B_2)\\
	&
	=\Per_s(B_1,\R^2)-\Ll_s(K,\Co B_2).
	}
	Similarly,
	$$
	\Per_s(K,\Omega)=\Ll_s(K, B_2\setminus K)=
	\Per_s(K,\R^2)-\Ll_s(K,\Co B_2)
	$$
	and
	$$
	\Per_s(B_2,\Omega)=
	\Ll_s(B_2\setminus K,\Co B_2)=
	\Per_s(B_2,\R^2)-\Ll_s(K,\Co B_2)=2^{2-s}\Per_s(B_1,\R^2)-\Ll_s(K,\Co B_2).
	$$
	These identities, together with~\eqref{eq:strict_esti_sets_per}, show the validity of~\eqref{eq:strict_comp}, concluding the proof.
\end{proof}

With this preliminary work, we can complete the proof of Theorem~\ref{pro42} by arguing as follows:

\begin{proof}[Proof of Theorem~\ref{pro42}]
In the setting of Proposition~\ref{pro41}, we consider~$\varphi:=\chi_K$ and~$u:=\chi_E$. We observe that~$\{u\ge\lambda\}$ is either empty, the whole space~$\R^2$, or equal to~$E$, which minimizes~$\Per_s(\,\cdot\,,\Omega)$ with respect to its datum outside~$\Omega$.
Hence, by~\cite[Theorem~1.3]{bdlv}, we have that~$u$ is an $s$-minimal function in~$\Omega$.

Even if~$\varphi=\chi_K\in C(\R^2\setminus\Omega)$, Proposition~\ref{pro41} yields that~$u=\chi_E\not\in C(\Omega)$, thus providing the desired example.
\end{proof}

\section{Proof of the comparison result in Theorem~\ref{cp1}}\label{fgiw4t734sec777}

To deal with the proof of Theorem~\ref{cp1},
we now build suitable maximal and minimal solutions. For this purpose, for all~$t\in \R$ we 
denote by
\[ \mathcal E_t :=\{ \varphi \geq t\}.\]

Then, there exists~$\Theta = \Theta(n,s)>1$ such that, if  
\[ \varphi \in L^\infty(\Omega_{\Theta \diam (\Omega)} \setminus \Omega), \]
there exist~$ E_t$ and~$  F_t$ which are the unique $s$-minimal sets in~$\Omega$
of maximum and minimum volume, respectively (see~\cite[Proposition~4.6]{bdlv}). 

We then denote 
\sys[
 \overline{u}(x) := ]{&\sup\{\theta\,{\mbox{ s.t. }} \, x\in \overline{E_\theta} \}, && {\mbox{if }}x\in \Omega, \\
 &\varphi(x), &&{\mbox{if }} x\in \Co \Omega,}
 and
 \sys[
 \underline {u}(x) := ]{&\sup\{\theta\, {\mbox{ s.t. }} \, x\in \overline{F_\theta} \}, && {\mbox{if }}x\in \Omega, \\
 &\varphi(x), && {\mbox{if }}x\in \Co \Omega,} and we refer to them as maximal and minimal solutions, respectively.

Moreover, we observe that 
 \[ \underline{u} \leq \overline{u} \qquad \mbox{ almost everywhere in } \, \Rn\]
 and that, by~\cite[Theorem~4.8]{bdlv}, both~$\overline{u}$ and~$ \underline{ u}$ are $s$-minimal functions belonging to~$L^\infty(\Omega)$. 
 
 Additionally, up to enlarging~$\Theta$, if~$\varphi\in C(\overline{\Omega_{\Theta \diam(\Omega)}}\setminus\Omega)$, then~$\underline{ u}$ and~$\overline{u}$ are continuous in~$\Omega$, and if~$\Omega$ is of class~$C^2$ then~$\underline{u}$ and~$\overline{u}$ can be extended with continuity to functions belonging to~$C (\overline{\Omega})$. 
 
With this, we can now proceed with the proof of Theorem~\ref{cp1}.

\begin{proof}[Proof of Theorem~\ref{cp1}]
By Theorem~\ref{exlev}, we have that~$\underline{u}_i$ and~$ \overline{u}_i$, with~$i\in \{1,2\}$, are continuous in~$\Omega$ (more precisely, we identify them with a continuous representative as in Proposition~\ref{kk1}).

We prove that~$\overline{u}_1\ge \overline{u}_2$ (the other claim being
similar).
Suppose by contradiction
that there exists~$x_0\in \Omega$ such that~$ \overline{ u}_1(x_0) <\overline{u}_2(x_0)$
and let~$t$ be such that
\[\overline{u}_1(x_0) <t<\overline{u}_2(x_0).\]
By continuity, there exists~$\eps>0$ such that~$B_\eps(x_0)\subset \Omega$  with
\eqlab{ \label{kl11} 	|\{ \overline{u}_1< t \}\cap B_\eps(x_0)| = |B_\eps(x_0)|
\qquad{\mbox{and}}\qquad
|\{ \overline{u}_2>t \}\cap B_\eps(x_0)| = |B_\eps(x_0)|. }
Notice that
\[ \E_t^2:=\{\varphi_2\geq t\} \subset \E_t^1:= \{ \varphi_1\geq t\} .\]
In the proof of~\cite[Theorem~4.8]{bdlv}, we have obtained that up to sets of measure zero,
\[ E_t^i=\{ \overline{u}_i \geq t \}, \] 
where~$E_t^i$ is the $s$-minimal set in~$\Omega$ with maximum volume, with respect to the exterior data~$E_t^i\cap \Co \Omega=\E_t^i$,  for~$i=1,2$.

Moreover, it holds that
\begin{equation}\label{fyrfhijedkof45689705}
\Per_s(E_t^1\cup E_t^2, \Omega) + \Per_s(E_t^1\cap E_t^2, \Omega) \leq \Per_s(E_t^1,\Omega)+ \Per_s(E_t^2,\Omega),\end{equation} 
see~\cite{MR3156889}. Since
$\left(E_t^1\cap E_t^2\right) \cap \Co \Omega =\E_t^2,$ we have that~$E_t^1\cap E_t^2$ is a competitor for~$E_t^2$, hence
$$ 
\Per_s(E_t^2, \Omega) \leq \Per_s(E_t^1\cap E_t^2, \Omega).$$
Plugging this information into~\eqref{fyrfhijedkof45689705}, we find that
\[\Per_s(E_t^1\cup E_t^2, \Omega) \leq \Per_s(E_t^1, \Omega).\]

Now, we have that~$(E_1\cup E_t^2) \cap \Co \Omega= \E_t^1$, 
hence~$E_t^1\cup E_t^2$ is $s$-minimal in~$\Omega$
with respect to the exterior datum~$\E_t^1$. Also, 
\[ \left(E_t^1\cup E_t^2\right)\cap \Omega= E_t^1\cap \Omega,\]
since~$E_t^1$ is of maximum volume. It follows that, up to sets of measure zero, $E_t^2\cap \Omega \subset E_t^1\cap \Omega$, thus contradicting~\eqref{kl11}.
\end{proof}

\begin{appendix}

\section{Continuous functions in the sense of Lebesgue}\label{appe:po8372}

This appendix elaborates on the proof of Theorem~\ref{exlev}. There, in order to obtain the continuity of an $s$-minimal function, we used the fact that, roughly speaking, the superlevel sets at different levels are properly contained one inside the other. We point out that such a property is not related to the $s$-minimality of a function, but rather it is a characterization of continuity. This is made precise in the following result.
Then, in the proof of Theorem~\ref{exlev} the $s$-minimality was used to ensure the validity of~\eqref{eq:bdary_level_sets}, through the strict maximum principle~\cite[Theorem~1.1]{seo23}.

\begin{proposition}\label{kk1}
Let~$\Omega$ be a bounded 
open  set. A measurable, locally essentially bounded function~$u\colon \Omega \to \R$ is continuous in~$ \Omega$, more precisely, there exists~$\widetilde u \colon \Omega \to \R$ such that~$\widetilde u = u$ almost everywhere in~$\Omega$ and~$\widetilde u \in C(\Omega)$,
if and only if, for all~$t\neq \tau \in \R$,
\begin{equation}\label{eq:bdary_level_sets} \partial^- \{ u \geq t\} \cap \partial^- \{u\geq \tau\}\cap\Omega =\emptyset.
	\end{equation} 
\end{proposition}

\begin{proof}
We remark that~$u$ is continuous in the precise sense given by Proposition~\ref{kk1} if and only if 
\[ \essliminf_{x\to x_0}u(x)= \esslimsup_{x\to x_0} u(x)\]
for all~$x_0\in \Omega$. Moreover, since~$u$ is locally essentially bounded, both these limits are well-defined real numbers.

Suppose now that~$u$ is continuous in~$\Omega$, and suppose by contradiction that for some~$t<\tau $ there exists 
 \[ x_0\in  \partial^- \{ u \geq t\} \cap \partial^- \{u\geq \tau\}\cap\Omega .\]
 Then, there exists~$r_0>0$ such that~$B_{r_0}(x_0)\subset\Omega$ and
 \[ |\{u <t\} \cap B_r(x_0)|>0
 \qquad{\mbox{and}}\qquad |\{u \geq \tau\} \cap B_r(x_0)|>0\] 
 for all~$r\in(0,r_0)$. Hence,
 \[ \essinf_{B_r(x_0)} u <t<\tau \le\esssup_{B_r(x_0)}\qquad\mbox{for all }r\in(0,r_0), \] 
 which contradicts the continuity of~$u$.
 
To prove the opposite, suppose by contradiction that~$u$ is not continuous. Then, there exists~$x_0\in  \Omega$ such that
\[  \essliminf_{x\to x_0} u(x) < \esslimsup_{x\to x_0} (x).\] 
For all~$t$ such that 
\[  \essliminf_{x\to x_0} u(x) <t <\esslimsup_{x\to x_0} (x),\] there exists~$r$ small enough such that
\[ \essinf_{B_r(x_0)} u(x) <t<\esssup_{B_r(x_0)} u(x)\] 
which implies that 
\[ x_0\in \partial^-\{u\geq t\}.\] 
Now consider
\[  \essliminf_{x\to x_0} u(x)<t_1<t_2< \esslimsup_{x\to x_0} u(x)\]
and observe that 
\[ x_0\in \partial^-\{u\geq t_1\} \cap \partial^-\{u\geq t_2\},\]
which gives the desired contradiction.  
\end{proof}


On a related note, it is convenient to point out the following well-known property of continuous functions, which was used in the end of the proof of Theorem~\ref{exlev}.

\begin{remark}\label{kk2}
Let~$\mathcal O\subset\Rn$ be a bounded connected 
open  set and let~$\varphi \colon \mathcal O\to \R $ be a continuous function. Then, for all $t_1\neq t_2 \in (\inf_{\mathcal O} \varphi,\sup_{\mathcal O}\varphi) $ it holds that
\begin{equation}\label{fio1} \{ \varphi \geq t_1 \}  \neq \{ \varphi \geq t_2 \} ,
\end{equation}
and moreover
\begin{equation}\label{fio2}|\{ \varphi \geq t_1 \} \Delta \{\varphi\geq t_2\}|>0.
\end{equation}
Indeed, since~$\varphi$ is continuous and~$\mathcal O$ is connected, we have that also~$\varphi(\mathcal O)$ is connected. This implies that
\[
\{\varphi=t\}\not=\emptyset\qquad\mbox{for any }t\in \Big(\inf_{\mathcal O} \varphi,\sup_{\mathcal O}\varphi\Big).
\]
Consider now two values~$t_1>t_2$ within~$(\inf_{\mathcal O} \varphi,\sup_{\mathcal O}\varphi)$. Then, we clearly have
\[
\{ \varphi \geq t_1 \}\subset\{ \varphi > t_2 \},
\]
and
\[
\emptyset\not=\{\varphi=t_2\}\subset\big(\{ \varphi \geq t_2 \}\setminus\{\varphi\geq t_1\}\big),
\]
thus \eqref{fio1} is proved. 
Moreover, considering~$\tilde{t}:=(t_1+t_2)/2$ and letting~$x\in\{\varphi=\tilde{t}\}$, by continuity we can find~$\delta>0$ such that~$B_\delta(x)\subset\{ \varphi \geq t_2 \}\setminus\{\varphi\geq t_1\}$, and we obtain the claim.

We also point out that if $\mathcal O$ is not connected, then the thesis does not necessarily hold. One can take for instance $\mathcal O= B_1(x)\cup B_1(y)$ with $|x-y|=3$ and $\varphi=0 $ on $B_1(x)$, $\varphi=3$ on $B_2(y)$. Then $\{\varphi \geq 1\} =\{\varphi\geq 2\}$. 
\end{remark}
\section{Some observations on the proofs of Theorem \ref{LACK}  and Proposition \ref{cont}}\label{apboh}

We recall the set function  $\alpha $, introduced in \cite{asympt1}, as 
\[ \alpha(E_0):=\lim_{s\to 0} s \int_{\Co B_1}\frac{\chi_{E_0}(x)}{|x|^{n+s}} \, dx ,\]
-- whenever such limit exist -- which is significant when dealing with the $s$-perimeter when $s\to 0$. Notice that such a limit may not exist even for smooth sets $E_0$, and this observation led in \cite{sticki} to define
\begin{equation}\label{alphaf} \overline \alpha(E_0)= \limsup_ {s\to 0} s \int_{\Co B_1}\frac{\chi_{E_0}(x)}{|x|^{n+s}} \, dx .
\end{equation} 
In \cite[Theorem 1.7]{sticki}, the authors proved that if $E_0$ does not completely  surround $\Omega$ and if  $\overline \alpha(E_0)$ is strictly smaller than $\omega_n/2$ (which is the function $\overline \alpha$ of the half-space), then for $s$ small enough, the only $s$-minimal set in $\Omega$ with respect to $E_0$ is the empty set. The next result establishes some sort of monotonicity with respect to the exterior data: there is some $s_0:=s_0(n, \Omega, E_0)$ such that for all $s<s_0$, if one considers the $s$-minimal set with respect to any subset of $E_0$, then still $E\cap \Omega=\emptyset$. The proof follows with a careful reading of \cite[Theorem 1.2, Theorem 1.7]{bdlv}, we insert a sketch for the reader's benefit.

\begin{proposition} \label{propb1} Let $\Omega\subset\Rn$ be a bounded and connected open set with $C^2$ boundary and let $E_0\subset \Co \Omega$ be be such that
\[ \overline \alpha(E_0)<\frac{\omega_n}2\] and such that there exists $p \in \partial \Omega$ and $r>0$ with
\[ B_r(p)\setminus \Omega \subset \Co E_0.\]  Then there exists $s_0:= s_0(n, \Omega, E_0)\in (0,1/2),$ such that for all $s\in (0, s_0)$, given any $ E_1\subset E_0$, the $s$-minimal set $E$ for the fractional perimeter with respect to  $ {E_1}$ is empty inside $\Omega$, i.e. 
\[ E\cap \Omega= \emptyset.\] 
\end{proposition}

\begin{proof} We let 
\[ R>2 \max\{1, \mbox{diam} (\Omega)\},\] 
and using \cite[Proposition 2.1]{sticki}, we first notice that
\bgs{\lim\inf_{s\to 0} \left({\omega_n} R^{-s} -2 s\sup_{x\in \overline \Omega}\int_{\Co B_R(x)}\frac{\chi_{E_0}(y)}{|x-y|^{n+s}} \, dy \right)= \omega_n - 2\overline \alpha(E_0):= 4\beta.}
Then there is some $s'=s'(n,E_0)$ such that for all $s<s'$,
\eqlab{ \label{fff} \frac{7\beta}2 \leq {\omega_n} R^{-s} -2 s\sup_{x\in \overline \Omega}\int_{\Co B_R(x)}\frac{\chi_{E_0}(y)}{|x-y|^{n+s}} \, dy \leq {\omega_n} R^{-s} -2 s\sup_{x\in \overline \Omega}\int_{\Co B_R(x)}\frac{\chi_{E_1}(y)}{|x-y|^{n+s}} \, dy. }
We prove at first that there is some $ s_0:=s_0(n,\Omega, E_0) \in (0, 1/2)$ such that if any set $E$ that coincides with $E_1\subset E_0$ outside of $\Omega$ has a tangent exterior ball of radius at least
\[ \delta_{ s_0}:= e^{-\frac{1}{ s_0} \log \frac{\omega_n+2\beta}{\omega_n+\beta}} \]
at some point $x\in \partial E \cap \overline \Omega$, then for all $s\in (0,s_0)$,
\eqlab{\label{boh1} \liminf_{\rho \to 0} H^\rho_s[E](x) \geq \frac{\beta}{s}.}
Using \eqref{mc} we write for $\rho>0$ small,
\[ H_s^\rho[E](x)= \int_{B_R(x) \setminus B_\rho(x)}\frac{\chi_{\Co E}(y)-\chi_E(y)}{|x-y|^{n+s}} \, dy +\int_{\Co B_R(x) } \frac{\chi_{\Co E}(y)-\chi_E(y)}{|x-y|^{n+s}} \, dy :=\mathcal I_s[E](x)+\mathcal J_s[E](x) .\] 
 With the same estimates as in the proof of \cite[Theorem 1.2]{sticki} -- see equation \cite[(3.3)]{sticki}, we have that there is $C_0:=C_0(n)>0$ such that on the one hand
\bgs{ \mathcal I_s[E](x)\geq \frac{1}{s}\left( \omega_n R^{-s}  -\omega_n \delta^{-s} -\frac{C_0s}{1-s}\delta^{-s}
\right)
,} 
and on the other hand
\bgs{ \mathcal J_s[E](x)\geq  \frac{\omega_n}s R^{-s} -2 \sup_{x\in \overline \Omega}\int_{\Co B_R(x)}\frac{\chi_{E_1}(y)}{|x-y|^{n+s}} \, dy .}
Using \eqref{fff} and that,  as $s\to 0$,
\[ R^{-s} \nearrow 1, \]
 we find that there is some $s_0:=s_0(n,\Omega,E_0)<s'$  such that for all $s\in (0, s_0)$ and all $\delta \geq \delta_{s_0}$,
\[ \omega_n R^{-s} \geq  \omega_n - \frac{\beta}2, \qquad  \delta^{-s} \leq \delta_{s_0}^{-s_0}=\frac{\omega_n +2\beta}{\omega_n+\beta}, \qquad\frac{C_0s}{1-s}\leq\beta, 
\]
and we obtain \eqref{boh1}. We reason by contradiction and suppose there is some boundary of $E$ inside $\Omega$.  We continue with the strategy of \cite[Theorem 1.7]{sticki}, without providing full technical details.  Notice at first that 
\[ B_r(p)\setminus \Omega \subset \Co E_1.\] We consider $s_0$ eventually smaller such that we can take a ball of radius at least $\delta_{s_0}$, tangent to $\partial \Omega$ at $p$ and contained in  $B_r(p)\setminus \Omega$. We  slide this ball along the interior normal to $\partial \Omega$ at $p$, until we first encounter $\partial E$. At the first contact point $x\in \partial E \cap  \overline \Omega$, we have the Euler-Lagrange equation
\[ H_s[E](x)=0,\]
which gives a contradiction to \eqref{boh1}. Once the ball, moved along the normal, is contained in $\Omega$, we can "move it around" all $\Omega$, excluding any contact point between the ball and $\partial E\cap \overline \Omega$.  
\end{proof}

\end{appendix}
\bibliography{biblio}
\bibliographystyle{plain}

\end{document}